\documentclass[a4paper,12pt]{article}
\usepackage{bbm}
\usepackage{amssymb, amsmath, amscd, amsfonts, amsthm}
\usepackage[body={17cm, 25cm}, centering, dvipdfm]{geometry}
\usepackage[tiny]{titlesec}
\usepackage[symbol]{footmisc}
\usepackage{mathrsfs}
\usepackage{enumerate}
\usepackage{graphicx}
\usepackage[square, numbers, sort&compress]{natbib}
 \newtheorem{thm}{Theorem}[section]
 \newtheorem{cor}[thm]{Corollary}
 \newtheorem{lem}[thm]{Lemma}
 
 \theoremstyle{remark}
 
 \newtheorem{exa}[thm]{Example}
 \numberwithin{equation}{section}
\title{\textbf{\large A Positive Solution for a Nonlocal Schr\"{o}dinger Equation}}

\author{\small Yongchao Zhang\footnote{School of Mathematics and Statistics, Northeastern University at Qinhuangdao, Taishan Road \#143, Qinhuangdao, China, 066004. E-mail: ldfwq@163.com.} \;and
Gaosheng Zhu\footnote{School of Science, Tianjin University, Weijin Road \#92, Tianjin, China, 300072. E-mail: gaozsm@163.com.}}
\date{}

\begin{document}
\maketitle
\vspace{-0.5cm}
\begin{abstract}
We provide an existence result of radially symmetric, positive, classical solutions for a nonlinear Schr\"{o}dinger equation driven by the infinitesimal generator of a rotationally invariant L\'{e}vy process.\\
\noindent\textit{Key Words}: nonlocal Schr\"{o}dinger equation; positive solution; mountain pass theorem.\\
\noindent\textit{Mathematics Subject Classification (2010)}: 35A01; 35A15; 35J60.\\
\end{abstract}

\section{Introduction}\label{intro}
The purpose of this paper is to provide an existence result of radially symmetric, positive, classical solutions for the following problem,
\begin{equation}\label{e:Nonloc}
\left\{
\begin{aligned}
&-2Au+\lambda u=|u|^{p-2}u\\
&u\in H^1(\mathbb{R}^N),
\end{aligned}
\right.
\end{equation}
where $\lambda>0$, $2\leq N \leq 6$, $2<p<2^*$ with $2^*:=+\infty$ if $N=2$ and $2^*:=2N/(N-2)$ if $N>2$, and $A$ is the infinitesimal generator of a rotationally invariant L\'{e}vy process.

\begin{exa}
Consider the infinitesimal generator $A$ of a L\'{e}vy process with jumps of normal distribution.
\[
Au(x):=\frac{1}{2}\Delta u(x)+\frac{1}{2}\int_{\mathbb{R}^N}(u(x+y)+u(x-y)-2u(x))\varphi(y)\mathrm{d}y,
\]
where $\varphi(y):=(2\pi)^{-N/2}\exp(-|y|^2/2)$.
\end{exa}

A basic motivation for the study of the problem (\ref{e:Nonloc}) is the well known nonlinear Schr\"{o}dinger equation driven by the infinitesimal generator of a Brownian motion,
\begin{equation}\label{e:SchBro}
-\Delta u+\lambda u=|u|^{p-2}u.
\end{equation}
Many authors investigated Equation (\ref{e:SchBro}) (see \cite{FloeWein1986, PinoFelm1996, ByeoWang2002, PinoFelm2002, JeanTana2005} etc.).

Note that the Brownian motion is a special rotationally invariant stable L\'{e}vy process. It is natural to consider the following equation,
\begin{equation}\label{e:SchSta}
(-\Delta)^{\alpha/2} u+\lambda u=|u|^{p-2}u,
\end{equation}
where $0<\alpha\leq 2$, since $-(-\Delta)^{\alpha/2}$ is the infinitesimal generator of a rotationally invariant stable L\'{e}vy process with index $\alpha$.
Equation (\ref{e:SchSta}) has been studied by many authors (see \cite{MousWeth2012, FelmAuaaTan2012, DipiPalaVald2012, FallVald2013} etc.).

Naturally, we consider the following (nonlocal) Schr\"{o}dinger equation,
\begin{equation}\label{e:SchLev}
-2A u+\lambda u=|u|^{p-2}u,
\end{equation}
where $A$ is the infinitesimal generator of a rotationally invariant L\'{e}vy process. In the present paper, we assume that the L\'{e}vy process is of $N$ dimensions, where $2\leq N\leq 6$, with nondegenerate diffusion terms and a finite L\'{e}vy measure.

Equation (\ref{e:SchLev}) also arises from looking for the standing waves of the following Schr\"{o}dinger equation,
\begin{equation*}
\mathrm{i}\frac{\partial\psi}{\partial\, t}-2A \psi=|\psi|^{p-2}\psi.
\end{equation*}

Before stating the main result of the present paper, let us make some comments on the operators $-(-\Delta)^{\alpha/2}$ and $A$. If $0<\alpha<2$, then the L\'{e}vy processes generated by $-(-\Delta)^{\alpha/2}$ are pure jump processes; in other words, these processes do not contain any diffusion term. In fact, the corresponding characteristics of them are given by $(0,\,0,\,\mu)$ with
\[
\text{$\mu(\mathrm{d}x)=\frac{K(\alpha) \mathrm{d}x}{|x|^{N+\alpha}}$ for some positive constant $K(\alpha)$.}
\]
Consequently, the L\'{e}vy measure $\mu$ is not finite. For the operator $A$, the corresponding characteristics are given by $(0,\,aI,\,\nu)$ for some positive number $a$ and some finite  rotationally invariant L\'{e}vy measure $\nu$. Therefore, $-(-\Delta)^{\alpha/2}$ does not cover operators of type $A$ and vice versa; besides, Equation (\ref{e:SchLev}) is an extension of Equation (\ref{e:SchBro}).

Now we state the the main result as follows.

\begin{thm}\label{t:MainTheo}
$\mathrm{(1)}$ Any weak solution of the problem (\ref{e:Nonloc}) in $H^1(\mathbb{R}^N)$ is a $C^2$ continuous function.

\noindent $\mathrm{(2)}$ There exists a radially symmetric, positive, classical solution of problem (\ref{e:Nonloc}).

\noindent $\mathrm{(3)}$ The values of any positive solution of the problem (\ref{e:Nonloc}) at maximum points are not less than $\lambda^{\frac{1}{p-2}}$.
\end{thm}

The rest of the paper is organized as follows. In Section \ref{Prelimi}, we present some preliminaries. The proof of Theorem \ref{t:MainTheo} will be given in Section \ref{ProofMainThm}.

\section{Some Preliminaries}\label{Prelimi}
This section serves as a preparation for the proof of Theorem \ref{t:MainTheo}. First, we state a compact embedding result. Second, a regularity result will be proved. Finally, we investigate the sign of solutions for a modified version of Equation (\ref{e:SchLev}).

Define
\[
\text{$H^1_{\mathbf{O}(N)}(\mathbb{R}^N):=\{u\in H^1(\mathbb{R}^N): u=gu, \;g\in\mathbf{O}(N)\}$, where $gu:=u\circ g^{-1}$.}
\]
Then we have the following lemma.
\begin{lem}[{\cite[p.18, Corollary 1.26]{Willem1996}}]\label{l:CompEmbe}
The following embedding is compact,
\[
H^1_{\mathbf{O}(N)}(\mathbb{R}^N)\hookrightarrow L^p(\mathbb{R}^N),\; 2<p<2^*.
\]
\end{lem}

\begin{lem}\label{l:Regula}
If $u$ is a weak solution of the equation
\[
-2Au+\lambda u=(u^+)^{p-1}
\]
in $H^1(\mathbb{R}^N)$, then $u\in C^2(\mathbb{R}^N)$.
\end{lem}
\begin{proof}
1. Note that the symbol $\sigma_A$ of $A$ is given by
\[
\sigma_A(\xi)=-\frac{a}{2}|\xi|^2+\int_{\mathbb{R}^N}[\cos(\xi\cdot x)-1]\nu(\mathrm{d}x),
\]
where $a$ is a positive number and $\nu$ is a finite $\mathbf{O}(N)$-invariant L\'{e}vy measure (see \cite[p.128, Exercise 2.4.23 and pp.163-164, Theorem 3.3.3]{App2009}).

Let $A_2$ be the operator with the symbol
\[
\sigma_{A_2}(\xi)=-\frac{a}{2}|\xi|^2,
\]
and $A_0$ be the operator with the symbol
\[
\sigma_{A_0}(\xi)=\int_{\mathbb{R}^N}[\cos(\xi\cdot x)-1]\nu(\mathrm{d}x).
\]

Then we have
\[
-2A_2u=h(\cdot)(1+|u|),
\]
where
\[
\textrm{$h(x):=\frac{2A_0u(x)+(u^+(x))^{p-1}-\lambda u(x)}{1+|u(x)|}$ \quad for $x\in\mathbb{R}^N$.}
\]

\noindent 2. For any $u\in H^1(\mathbb{R}^N)$, we have
\begin{equation}\label{e:ZeroOder}
\int_{\mathbb{R}^N}(1+|\xi|^2)\left(\int_{\mathbb{R}^N}[\cos(\xi\cdot x)-1]\nu(\mathrm{d}x)\right)^2|\widehat{u}(\xi)|^2\mathrm{d}\xi<\infty,
\end{equation}
where ``\;\,$\widehat{}$\;\,'' denotes the Fourier transformation.

Thus $A_0:H^1(\mathbb{R}^N)\rightarrow H^1(\mathbb{R}^N)$ is a bounded operator thanks to (\ref{e:ZeroOder}).

Furthermore, it follows that $h\in L_{\mathrm{loc}}^{N/2}(\mathbb{R}^N)$. Consequently, we have $u\in L_{\mathrm{loc}}^q(\mathbb{R}^N)$ for any $q\in[1, +\infty)$ by Br\'{e}zis-Kato theorem (see, for example, \cite[p.270, B.3 Lemma]{Stru2008}). Then, by the ellipticity of operator $A$, we find $u\in W_{\mathrm{loc}}^{2,\, q}(\mathbb{R}^N)$ for any $q\in[1, +\infty)$. Now Sobolev embedding theorem implies $u\in C _{\mathrm{loc}}^1(\mathbb{R}^N)$. Finally, also by the ellipticity of operator $A$, it follows that  $u\in C^2(\mathbb{R}^N)$.
\end{proof}

\begin{lem}\label{l:Posi}
If $u\in C^2(\mathbb{R}^N)\cap H^1(\mathbb{R}^N)$ is a nontrivial solution of the equation
\[
-2Au+\lambda u=(u^+)^{p-1},
\]
then $u>0$.
\end{lem}
\begin{proof}
1. First we have
\[
\begin{split}
&\quad\int\int(u(x)-u(x+y))(u^-(x)-u^-(x+y))\nu(\mathrm{d}y)\mathrm{d}x\\
&=\int\int(u(x)-u(y))(u^-(x)-u^-(y))\nu(-x+\mathrm{d}y)\mathrm{d}x\leq 0,
\end{split}
\]
where we have used
\[
\begin{split}
\mathbb{R}^2=&\{x:u(x)\geq 0\}\times\{y:u(y)\geq 0\}\cup\{x:u(x)\geq 0\}\times\{y:u(y)< 0\}\\
&\cup\{x:u(x)< 0\}\times\{y:u(y)\geq 0\}\cup\{x:u(x)< 0\}\times\{y:u(y)<0\}
\end{split}
\]
for the inequality.

Then it follows that
\[
(-2Au,-u^-)_{L^2}=a\|\nabla u^-\|_{L^2}-\int\int(u(x)-u(x+y))(u^-(x)-u^-(x+y))\nu(\mathrm{d}y)\mathrm{d}x\geq 0.
\]

Therefore, in light of $(-2Au,-u^-)_{L^2}+\lambda\|u^-\|_{L^2}^2=0$, we have $u^-=0$, which implies $u\geq 0$.

\noindent 2. Rewrite the equation $-2Au+\lambda u=(u^+)^{p-1}$ as
\[
-2A_2u+(\lambda +2\nu(\mathbb{R}^N))u=(u^+)^{p-1}+2\int_{\mathbb{R}^N}u(\cdot+y)\nu(\mathrm{d}y).
\]
Then we find that
\[
-2A_2u+(\lambda +2\nu(\mathbb{R}^N))u\geq 0.
\]
It follows from the strong maximum principle that $u>0$.
\end{proof}

\begin{cor}
Assume that $u\in C^2(\mathbb{R}^N)\cap H^1(\mathbb{R}^N)$ is a nontrivial solution of the equation $-2Au+\lambda u=(u^+)^{p-1}$. If $x_0\in\mathbb{R}^N$ is a maximum point of the function $u$, then $u(x_0)\geq \lambda^{\frac{1}{p-2}}$.
\end{cor}
\begin{proof}1. Since $x_0$ is a maximum point of the function $u$, we have $\Delta u(x_0)\leq 0$.

\noindent 2. Note that Lemma \ref{l:Posi} implies $u(x_0)>0$. It follows from the positive maximum principle (see, for example, \cite[p.283, (1.5) proposition]{RevuYor1999} or \cite[p.181, Theorem 3.5.2]{App2009}) that $A_0u(x_0)\leq 0$. This and $\Delta u(x_0)\leq 0$ imply $Au(x_0)\leq 0$. Therefore,
\[
u(x_0)^{p-1}-\lambda u(x_0)=-2Au(x_0)\geq 0.
\]
So the inequality $u(x_0)\geq \lambda^{\frac{1}{p-2}}$ holds.
\end{proof}

\section{Proof of Theorem \ref{t:MainTheo}}\label{ProofMainThm}
In this section, we provide a proof of Theorem \ref{t:MainTheo} via the mountain pass theorem.

Observe that the operator $-A$ is positively self-adjoint (see \cite[p.178, Theorem 3.4.10 and p.190, Theorem 3.6.1]{App2009}). We define a new inner product on $H^1(\mathbb{R}^N)$ by
\[
\text{$(v,w):=(-2Av, w)_{L^2}+\lambda(v,w)_{L^2}$ for any $v,w\in C^{\infty}_0(\mathbb{R}^N)$,}
\]
and denote the induced norm of it by $\|\cdot\|$. Since the operator $-A_0$ is also positively self-adjoint, it follows from $A=A_2+A_0$ and (\ref{e:ZeroOder}) that the norm $\|\cdot\|$ is equivalent to $\|\cdot\|_{H^1}$.

Define a functional $E: H^1(\mathbb{R}^N)\rightarrow \mathbb{R}$ by
\begin{equation*}\label{e:Func}
E(u):=\frac{1}{2}\|u\|^2-\frac{1}{p}\int_{\mathbb{R}^N}(u^+(x))^p\mathrm{d}x.
\end{equation*}
Then it follows from \cite[p.11, Corollary 1.13]{Willem1996} that $E\in C^2(H^1(\mathbb{R}^N), \mathbb{R})$. In addition, the critical points of the functional $E$ are weak solutions of the equation $-2Au+\lambda u=(u^+)^{p-1}$ in $H^1(\mathbb{R}^N)$, and vice versa.

\begin{lem}
The functional $E$ is $\mathbf{O}(N)$-invariant.
\end{lem}
\begin{proof}
We only need to prove that the norm $\|\cdot\|$ is $\mathbf{O}(N)$-invariant.

Note that the symbol $\sigma_A$ of $A$ is given by
\[
\sigma_A(\xi)=-\frac{a}{2}|\xi|^2+\int_{\mathbb{R}^N}[\cos(\xi\cdot x)-1]\nu(\mathrm{d}x),
\]
where $a$ is a positive number and $\nu$ is a finite $\mathbf{O}(N)$-invariant L\'{e}vy measure (see \cite[p.128, Exercise 2.4.23 and pp.163-164, Theorem 3.3.3]{App2009}). We find the symbol $\sigma_A$ of $A$ is $\mathbf{O}(N)$-invariant.

Therefore, for any $\varphi\in C^{\infty}_0(\mathbb{R}^N)$ and $g\in \mathbf{O}(N)$, we have
\[
\begin{aligned}
\|g\varphi\|^2&=(-2A(g\varphi), g\varphi)_{L^2}+\lambda\|g\varphi\|_{L^2}^2\\
&=(-2\sigma_A\cdot\widehat{g\varphi}, \widehat{g\varphi})_{L^2}+\lambda\|g\varphi\|_{L^2}^2\\
&=(-2g^{-1}\sigma_A\cdot \widehat{\varphi}, \widehat{\varphi})_{L^2}+\lambda\|g\varphi\|_{L^2}^2\\
&=(-2\sigma_A\cdot \widehat{\varphi}, \widehat{\varphi})_{L^2}+\lambda\|\varphi\|_{L^2}^2=\|\varphi\|^2,
\end{aligned}
\]
which implies that the norm $\|\cdot\|$ is $\mathbf{O}(N)$-invariant.
\end{proof}

We need the following Lemma \ref{l:SupOperInfi} in the verification of the PS condition for the functional $E$ restricted to $H_{\mathbf{O}(N)}^1(\mathbb{R}^N)$.
\begin{lem}[{\cite[p.134, Theorem A.4]{Willem1996}}]\label{l:SupOperInfi}
Assume that $1\leq p<\infty$, $1\leq q<\infty$, and $g\in C(\mathbb{R}^N)$ such that
\[
\text{$|g(u)|\leq c|u|^{p/q}$ for some constant $c$.}
\]
Then the operator $L: L^p(\mathbb{R}^N)\rightarrow L^q(\mathbb{R}^N)$ defined by $u\mapsto g(u)$ is continuous.
\end{lem}

\begin{lem}[The PS condition for the functional $E$ restricted to $H_{\mathbf{O}(N)}^1(\mathbb{R}^N)$]\label{l:PSCondi}
Any sequence $\{u_n\}_{n\in\mathbb{N}}\in H_{\mathbf{O}(N)}^1(\mathbb{R}^N)$ such that
\[
d:=\sup\limits_{{}^{n\in\mathbb{N}}}\{E(u_n)\}<\infty,\quad E\,'(u_n)\rightarrow 0 \; \mathrm{as}\;n\rightarrow\infty
\]
contains a convergent subsequence.
\end{lem}
\begin{proof}
The proof is the same as that of \cite[p.15, Lemma 1.20]{Willem1996}.

\noindent 1. For $n$ large enough, we have
\[
d+\|u_n\|\geq E(u_n)-\frac{1}{p}\langle E\,'(u_n), u_n\rangle=\left(\frac{1}{2}-\frac{1}{p}\right)\|u_n\|^2.
\]
It follows that $\{u_n\}_{n\in\mathbb{N}}$ is bounded in $ H_{\mathbf{O}(N)}^1(\mathbb{R}^N)$ since $p>2$.

\noindent 2. Without loss of generality, we assume that $u_n\rightharpoonup u$ in $H_{\mathbf{O}(N)}^1(\mathbb{R}^N)$. Then it follows from Lemma \ref{l:CompEmbe} that $u_n\rightarrow u$ in $L^p(\mathbb{R}^N)$. Consequently, by Lemma \ref{l:SupOperInfi}, we have $(u_n^+)^{p-1}\rightarrow (u^+)^{p-1}$ in $L^q(\mathbb{R}^N)$, where $q:=p/(p-1)$.

Note that
\begin{equation}\label{e:pscond}
\|u_n-u\|^2=\langle E\,'(u_n)-E\,'(u), u_n-u\rangle+\int_{\mathbb{R}^N}({u^+_n(x)}^{p-1}-{u^+(x)}^{p-1})(u_n(x)-u(x))\mathrm{d}x.
\end{equation}

For the first term of the right hand side of the above equality, we see that
\[
\langle E\,'(u_n)-E\,'(u), u_n-u\rangle\rightarrow 0\;\mathrm{as}\;n\rightarrow\infty,
\]
since  $E\,'(u_n)\rightarrow 0 \; \mathrm{as}\;n\rightarrow\infty$ and $\{u_n\}_{n\in\mathbb{N}}$ is bounded in $ H_{\mathbf{O}(N)}^1(\mathbb{R}^N)$ .

And for the second term, it follows from H\"{o}lder inequality that
\[
\begin{aligned}
&\quad\int_{\mathbb{R}^N}({u^+_n(x)}^{p-1}-{u^+(x)}^{p-1})(u_n(x)-u(x))\mathrm{d}x\\
&\leq\|{u^+_n(x)}^{p-1}-{u^+(x)}^{p-1}\|_{L^q}\|u_n(x)-u(x)\|_{L^p}\rightarrow 0\;\mathrm{as} \;n\rightarrow\infty,
\end{aligned}
\]
because $u_n\rightarrow u$ in $L^p(\mathbb{R}^N)$ and $(u_n^+)^{p-1}\rightarrow (u^+)^{p-1}$ in $L^q(\mathbb{R}^N)$.

Therefore, $u_n\rightarrow u$ in $H_{\mathbf{O}(N)}^1(\mathbb{R}^N)$ as $n\rightarrow\infty$ by (\ref{e:pscond}).
\end{proof}

Now we are at the position to give a proof of Theorem \ref{t:MainTheo}.
\begin{proof}[Proof of Theorem \ref{t:MainTheo}]
1. Consider the functional $E$ restricted to $H_{\mathbf{O}(N)}^1(\mathbb{R}^N)$. Thanks to Lemma \ref{l:CompEmbe} or Sobolev imbedding theorem, there is a positive constant $c$ such that $\|u\|_{L^p}\leq c\|u\|$ for any $u\in H_{\mathbf{O}(N)}^1(\mathbb{R}^N)$. Then it follows from the definition of the functional $E$ that
\[
E(u)\geq \frac{1}{2}\|u\|^2-\frac{c^p}{p}\|u\|^p.
\]
Setting $r:=\left(\frac{p}{4c^p}\right)^{\frac{1}{p-2}}$, we have
\[
\inf\limits_{{}^{\|u\|=r}}E(u)\geq \frac{1}{4}\left(\frac{p}{4c^p}\right)^{\frac{2}{p-2}}>0.
\]

\noindent 2. Set $w(x):=\exp\left(-|x|^2\right)$. Then $w(x)\in H_{\mathbf{O}(N)}^1(\mathbb{R}^N)$ and for any $t\in [0,+\infty)$,
\[
E(tw)=\frac{t^2}{2}\|w\|^2-\frac{t^p}{p}\|w\|_{L^p}^p.
\]

Note that $p>2$. We can take a positive number $t$ such that $t\|w\|>r$ and $E(tw)<0$.

\noindent 3. Now by the mountain pass theorem, there is a nontrivial critical point $u$ of the functional $E$ restricted to $H_{\mathbf{O}(N)}^1(\mathbb{R}^N)$. Note that the functional $E$ is $\mathbf{O}(N)$-invariant. Thanks to the principle of symmetric criticality (see, for example, \cite[p.18, Theorem 1.28]{Willem1996}), it follows that the point $u$ is also a critical point of the functional $E$. Consequently, the point $u$ is a weak solution of the equation $-2Au+\lambda u=(u^+)^{p-1}$ in $H^1(\mathbb{R}^N)$.

\noindent 4. Finally, Lemma \ref{l:Regula} and Lemma \ref{l:Posi} complete the proof.
\end{proof}

\bibliographystyle{amsplain}
\bibliography{Xbib}

\bigskip
\bigskip
%The authors declare that there is no conflict of interests regarding the publication of this article.

\end{document}